\documentclass[smallcondensed]{amsart}  

\usepackage{graphicx}

\usepackage{amsfonts, amssymb,amsmath}
\usepackage[mathscr]{eucal}

\usepackage{tikz}

\usetikzlibrary{arrows}
\usetikzlibrary{matrix}

\theoremstyle{plain}
\newtheorem{theorem}{Theorem}[section]
\newtheorem{lemma}[theorem]{Lemma}

\newtheorem{remark}[theorem]{Remark}
\newtheorem{example}[theorem]{Example}

\theoremstyle{definition}
\newtheorem{definition}[theorem]{Definition}

\newcommand{\prob}{\mathcal{P}}
\newcommand{\px}{\mathscr PX}
\newcommand{\be}{\begin{equation}} 
\newcommand{\ee}{\end{equation}}

\newcommand{\sa}{\Sigma}

\newcommand{\mO}[1]{(#1,\Sigma_{#1})}

\newcommand{\mcS}{\mathcal{S}}
\newcommand{\mcI}{\mathcal{I}}
\newcommand{\B}{\mathcal{B}}

\newcommand{\M}{\mathcal{M}_{cg}}
\newcommand{\meas}{\mathcal{M}}
\newcommand{\measT}{\M^{T}}

\begin{document}

\title{A Categorical Foundation for Bayesian Probability}

\author{Jared Culbertson         \and
        Kirk Sturtz
}


\maketitle

\begin{abstract}
Given two measurable spaces $H$ and $D$ with countably generated $\sigma$-algebras, a perfect prior probability measure $P_H$ on $H$ and a sampling distribution $\mcS\colon  H \rightarrow D$, there is a corresponding inference map $\mcI\colon  D \rightarrow H$ which is unique up to a set of measure zero. Thus, given a data measurement $\mu\colon 1 \rightarrow D$,  a posterior probability $\widehat{P_H}=\mcI \circ \mu$ can be computed.  This procedure is iterative: with each updated probability $P_H$, we obtain a new joint distribution which in turn yields a new inference map $\mcI$ and the process repeats with each additional measurement.    The main result uses an existence theorem for regular conditional probabilities by Faden, which holds in more generality than the setting of Polish spaces. This less stringent setting then allows for non-trivial decision rules (Eilenberg--Moore algebras) on finite (as well as non finite) spaces, and also provides for a common framework for decision theory and Bayesian probability. 
\vspace{5pt}

\noindent
{\em NB:} This paper has been published in Applied Categorical Structures,\\
 http://link.springer.com/article/10.1007/s10485-013-9324-9, please contact the authors at jared.culbertson@us.af.mil for the correct reference.

\end{abstract}

\section{Introduction}

Bayesian probability is a subject that has proven very successful in prediction, inference and model selection~\cite{Berger:Statistical_decision_theory,Howson-Urbach:scientific_reasoning,Jaynes:probability}.   
 \v{C}encov~\cite{Cencov:statistical_decision_rules} gives a categorical foundation for non-Bayesian statistical inference, but as far as the authors are aware, a categorical framework for Bayesian probability has not been fully developed. Lawvere took the first steps in this direction by defining the category of probabilistic mappings in the unpublished manuscript~\cite{Lawvere:prob_mappings}.  Following this, Lawvere
 and Huber~\cite{Lawvere:personal_2011} gave a seminar in Zurich on \emph{Bayesian Sections}, further developing this category as a basis for Bayesian probability.  The first appearance in the literature was an expansion of these ideas by Giry~\cite{Giry:categorical_probability}, who showed that the endofunctor on the category of measurable spaces $T\colon  \meas \to \meas$ associated to the probability adjunction given by Lawvere forms a monad, and that Lawvere's category of probabilistic mappings is the Kleisli category of that monad.  
 
 Subsequently, Meng~\cite{Meng:metric_spaces}, examined the category of convex sets and affine linear maps, which can be shown to be equivalent to the category of Eilenberg--Moore algebras of the Giry monad. This category can be thought of as the category of ``decision rules'' since the objects of that category are certain measurable functions $TX \to X$ whose fibers partition the space of probability measures on a given space $X$ into positive convex measurable sets.\footnote{Doberkat~\cite{Doberkat:tracing} refers to this process of making decision rules as \emph{derandomization}, which in certain applications like probabilistic semantics may be a more appropriate terminology.}  Based on the work by Giry restricting the monad to Polish spaces, Doberkat\cite{Doberkat:Kernel,Doberkat:stochastic_relations} has since characterized the Eilenberg--Moore $T$-algebras for these topological spaces.  That work, however, was based upon giving the space of probability measures the $\sigma$-algebra generated by the weak topology (as used for the Polish space monad) which results in (nontrivial) finite spaces having no $T$-algebras.  In the final section, we show that this negative result can be circumvented by avoiding topological conditions and using the initial $\sigma$-algebra generated by the evaluation maps. Others, including Wendt~\cite{Wendt:disintegrations}, van Breugel~\cite{vanBreugel:metric_monad} and Abramsky--Blute--Panangaden~\cite{Abramsky-Blute-Panangaden:nuclear_ideals} have also studied similar constructions.
 
Central to the study of Bayesian probability is the existence of regular conditional probabilities. Many textbooks restrict to Polish spaces in order to prove this existence ({\em e.g.}, see \cite{Dudley:probability}), though this is not a strictly necessary condition. Several more general characterizations of conditions which guarantee the existence of regular conditional probabilities have been found, either restricting the spaces involved or the joint distributions which are allowed. In \cite{Pachl:disintegrations}, Pachl does not require even countably generated $\sigma$-algebras, but relies instead on a certain notion of compactness. We will prefer to follow \cite{Faden:regular_conditional_probabilities}, where Faden gives a necessary and sufficient condition when we restrict to countably generated spaces. Namely, the marginals of a joint distribution must give {\em perfect} measure spaces. The class of perfect measures is broad and includes, for example, all Radon measures. For this reason, we will only consider perfect measures and begin by showing that the Giry monad restricted to perfect probability measures is still a monad (this is straightforward based on Theorem~\ref{theorem::perfect_measures}). Note that everything prior to Section~\ref{section::constructing_regular_conditionals} holds without restricting to perfect probability measures (see \cite{Giry:categorical_probability}).

The main theorem (Theorem~\ref{theorem::inference}) states that inference maps are uniquely determined by a prior probability and a sampling distribution. This result follows from the existence of regular conditional probabilities, and we restate an existence theorem (Theorem~\ref{theorem::regular_conditionals}) of Faden~\cite{Faden:regular_conditional_probabilities}. Using our characterization of Bayesian probability and the well-known fact that the Kleisli category embeds into the category of $T$-algebras, we can then see that the category of decision rules provides a common framework for both decision theory and Bayesian probability.  Some of these ideas are similar in spirit to the general notion of distributions based on commutative monads found in the recent paper of Kock~\cite{Kock:commutative_monads}. 

\section{The Category of Perfect Probabilistic Mappings}
We begin with an overview of the category of perfect probabilistic mappings, which is a slight modification of Lawvere's category (see \cite{Lawvere:prob_mappings}) of probabilistic mappings. The restrictions to have countably generated spaces as objects and restrict to perfect probability measures in the definition of the morphisms are required to ensure the existence of regular conditional probabilities, but are otherwise unnecessary. We include these restrictions throughout the paper to avoid overcomplicating the statements of the main theorems. See the introduction for a discussion of alternative definitions. Most of the following fundamental results are widely known without the restriction to perfect measures, and we indicate where the usual arguments must be modified due to this additional constraint. We also review the definition and foundational properties of perfect probability measures. 

Let us fix the following notation. We will denote the $\sigma$-algebra of a measurable space $X$ by $\sa_{X}$, and the category of countably generated measurable spaces and measurable functions by $\M$. For an  object $(X, \sa_X)$ in $\M$ we will often drop the associated $\sigma$-algebra from the notation and denote it simply by $X$ when the $\sigma$-algebra is obvious or inconsequential. We will use $(1, \sa_{1})$ and $(2, \sa_{2})$ for the one-element and two-element measurable sets with the discrete $\sigma$-algebras, but we will similarly just write ``$1$'' or ``$2$'' when these are used as objects in some category. 

\begin{definition}
A measure space $(X, \sa, \mu)$ is called {\em perfect} if for any measurable function $f \colon  X \to \mathbb{R}$, there exists a Borel set $E \subset f(X)$ such that $\mu(f^{-1}(E)) = \mu(X)$. A family $\{\mu_i\}_{i \in I}$ of measures on $X$ is {\em equiperfect} if given $f$ as before, there exists a single Borel set $E \subset f(X)$ with $\mu_{i}(f^{-1}(E)) = \mu_{i}(X)$ for all $i \in I$.
\end{definition}

The following theorem collects many basic results about perfect probability measures. We refer the reader to \cite{Rodine:perfect_probabilities} and \cite{Faden:regular_conditional_probabilities} and the references therein for proofs and more details on perfect measures. 

\begin{theorem}
\label{theorem::perfect_measures}
Let $(X, \sa_{X})$ and $(Y, \sa_{Y})$ be measurable spaces. Then
\begin{enumerate}
	\item[(a)] if $P$ is a \{0,1\}-valued probability measure, then $P$ is perfect
	\item[(b)] if $P$ and $Q$ are probability measures on $X$ such that $Q \ll P$, then $Q$ is perfect,
	\item[(c)] if $\sa_{X}' \subset \sa_{X}$ and a probability measure $P$ on $X$ is perfect, then $P$ restricted to $\sa_{X}'$ is perfect,
	\item[(d)] if $f\colon X \to Y$ is a measurable function and a probability measure $P$ on $X$ is perfect, then $f_{\ast}P$ is a perfect probability measure on $Y$,
	\item[(e)] if $J$ is a probability measure on $(X\times Y, \sa_{X}\otimes \sa_{Y})$ with marginals $P$ and $Q$, then $J$ is perfect if and only if $P$ and $Q$ are perfect,
	\item[(f)] if $f\colon X\times \sa_{Y} \to [0,1]$ is measurable for each fixed $B \in \sa_{Y}$ and $P$ is a perfect probability measure on $X$, then the probability measure on $Y$ defined by $Q(B) = \int_{X} Q(x,B) \,dP$ is perfect if and only if  $\{f(x, \cdot)\}_{x \in X}$ is an equiperfect family of probability measures on $Y$ except on a $P$-null set.
\end{enumerate}
\end{theorem}

\begin{definition}
The {\em category of perfect probabilistic mappings} $\prob$ has countably generated measurable spaces $(X, \sa_X)$ as  objects and an arrow between two such objects $f\colon (X, \sa_X)  \to (Y, \sa_Y)$ consists of a function $f\colon X \times \sa_Y  \to [0,1]$ such that
\begin{enumerate}
\item[(i)] for all $B\in \sa_Y$, the function $f(\cdot,B)\colon X \rightarrow [0,1]$ is measurable,
\item[(ii)] the collection $\{f(x,\cdot)\colon \sa_Y \rightarrow [0,1]\}_{x \in X}$ is an equiperfect family of probability measures on $Y$.
\end{enumerate}
That is, morphisms in $\prob$ can be thought of as parametrized families of perfect probability measures that vary measurably. For an arrow $f\colon \mO{X} \rightarrow \mO{Y}$ we will often denote the function $f(\cdot,B)\colon X \to [0,1]$ by $f_B$ and the function $f(x,\cdot)\colon \Sigma_Y \to [0,1]$ by $f_x$. 

Given two arrows
\begin{equation}
\mO{X}  \xrightarrow{f} \mO{Y}
\xrightarrow{g} \mO{Z}
\end{equation}
the composition $g \circ f\colon X \times \sa_Z  \rightarrow [0,1]$ is defined by
\begin{equation}
(g \circ f)(x,C) = \int_{y \in Y} g_{C}(y) df_x.
\end{equation}
This composition is well-defined due to Theorem~\ref{theorem::perfect_measures} (f), and associativity follows easily from the monotone convergence theorem.  An important fact is that every measurable function $f\colon X  \to Y$ may be regarded as a $\prob$-morphism $\delta_{f}\colon X  \to Y$, where the Dirac (or one point) measure
\begin{equation}
\delta_f(x,B) =
\begin{cases}
 1 & \text{if } f(x)\in B\\
 0 & \text{if } f(x)\notin B
\end{cases}
\end{equation}
assigns to each $x \in X$ the Dirac measure on $Y$ which is concentrated at $f(x)$.  Taking the measurable function $f$ to be the identity map on a particular measurable space $X$ gives the arrow $\delta_{Id_{X}}\colon \mO{X} \to \mO{X}$, {\em i.e.}, the identity arrow for $X$ in $\prob$. In fact, it is easy to check that the association $f \mapsto \delta_{f}$ determines a functor $\delta\colon  \M \to \prob$ taking a measurable space to itself. Note that this functor is not faithful, however, and so we do not get an embedding of $\M$ into $\prob$. We will call a $\prob$ arrow $P\colon X \to Y$ {\em deterministic} if for every $B \in \sa_Y$ the measurable functions $P_{B}\colon  X \to [0,1]$ assume only the values  $0$ or $1$. In fact, every deterministic $\prob$-arrow $P\colon X \to Y$ is of the form $P = \delta_{f}$ for some measurable function $f$, provided that $X$ has cardinality below the cardinality of the set of all measurable sets. 
\end{definition}

The following lemma gives two useful properties which follow easily from standard exercises in measure theory and the definition of composition in $\prob$.
\begin{lemma}  \label{lem::Dirac_property}   If $p\colon X \to Y$ is a measurable function and $f\colon Y \to Z$ a $\prob$-morphism, then the composition
\begin{equation*} 
 \begin{tikzpicture}[baseline=(current bounding box.center)]
	\node	(X)	at	(-3,0)	                  {$X$};
	\node        (Y)  at      (0,-0)                            {$Y$};
	\node        (Z)   at   (3,0)                 {$Z$};
	\draw[->, above] (X) to node  {$\delta_p$} (Y);
	\draw[->,above] (Y) to node {$f$} (Z);
 \end{tikzpicture}
 \end{equation*}
 is given by  $(f \circ \delta_p)(x,C)=f_{p(x)}(C)$. On the other hand, if $q\colon Y \to Z$ is a measurable function and $g\colon X \to Y$ a $\prob$-morphism, then the composition
 \begin{equation*}
 \begin{tikzpicture}[baseline=(current bounding box.center)]
	\node	(X)	at	(-3,0)	                  {$X$};
	\node        (Y)  at      (0,-0)                            {$Y$};
	\node        (Z)   at   (3,0)                 {$Z$};
	\draw[->, above] (Y) to node  {$\delta_q$} (Z);
	\draw[->,above] (X) to node {$g$} (Y);
 \end{tikzpicture}
 \end{equation*}
 is given by $( \delta_q \circ g)(x,C)=g_x(q^{-1}C)$.
\end{lemma}

There are several distinguished objects in $\prob$ that play an important role in many constructions. Any set $X$ with the indiscrete $\sigma$-algebra $\sa_X=\{X, \emptyset\}$ is a terminal object since any arrow $P\colon Y \rightarrow X$
is completely determined by the fact that $P_{y}$ must be a probability measure on $X$. We denote the canonical terminal object by $1$ since it is isomorphic to the one-element set. Notice that an arrow $P\colon 1 \to X$ is precisely a perfect probability measure on $X$ and that $1$ is a separator for $\prob$. In addition to having a separator, the category $\prob$ also has a coseparator, the two-element set $2=\{\top,\bot\}$ with $\sa_2$ the discrete algebra on $2$. Moreover, there is a set bijection $\hom_{\prob}(X,2) \simeq \hom_{\M}(X,[0,1]).$

We briefly show how the Giry monad factors through $\prob$. Let $\mathscr{P}X$ denote the set of perfect probability measures on $X$, endowed with the coarsest $\sigma$-algebra such that the evaluation maps $ev_{B}\colon  \mathscr PX \to [0,1]$ given by $ev_{B}(P) = P(B)$ are measurable.  Then we can define a functor $\mathscr P \colon  \prob \to \M$ which sends a measurable space $X$ to the space $\mathscr PX$ of probability measures on $X$. On arrows, $\mathscr P$ sends the $\prob$-arrow $f\colon  X \to Y$ to the measurable function $\mathscr Pf\colon  \mathscr PX \to \mathscr PY$ defined pointwise on $\sa_{Y}$ by
\be
	\mathscr Pf(P)(B) = \int_{X} f_{B} \, dP.
\ee
That is, $\mathscr Pf(P)$ gives the probability measure on $Y$ defined by the composition 
\be
\begin{tikzpicture}
	\node (1) at (0,0) {$1$};
	\node (X) at (2,0) {$X$};
	\node (Y) at (4,0) {$Y$};
	
	\draw[->, above] (1) to node {$P$} (X);
	\draw[->, above] (X) to node {$f$} (Y);
	\draw[->, below, out=345, in=195] (1) to node {$f \circ P$} (Y);
\end{tikzpicture}
\ee
in $\prob$. Since $\mathscr PX = \hom_{\prob}(1, X)$ as sets, another common notation for $\mathscr PX$ is $X^{1}$, but we will use the functor notation for clarity.

In fact, $\mathscr PX$ is an important object in $\prob$ as well, and based on the definition of the $\sigma$-algebra on $\mathscr PX$, we can define the evaluation morphism $\varepsilon_{X}\colon  \mathscr PX \to X$ by $\varepsilon_{X}(P, A) = P(A)$. With this, we are able to characterize the relationship between $\M$ and $\prob$, first proved in \cite{Giry:categorical_probability} for general measurable spaces and probability measures. To see that this proof goes through when restricting to countably generated spaces and perfect measures, we only need to observe that the $\sigma$-algebra defined above for $\mathscr PX$ is countably generated when $X$ is countably generated and that pushforwards take perfect measures to perfect measures. 
\begin{theorem} \label{adjunction} 
The functors $\delta\colon  \M \to \prob$ and $\mathscr P\colon  \prob \to \M$ form an adjunction
\[
\begin{tikzpicture}[baseline = (current bounding box.east)]
	\node (meas) at (0,0) {$\M$};
	\node (P) at (3,0) {$\prob$};
	
	\draw[->, above] ([yshift=2pt] meas.east) to node {$\delta$} ([yshift=2pt] P.west);
	\draw[->, below] ([yshift=-2pt] P.west) to node {$\mathscr P$} ([yshift=-2pt] meas.east);

	\node at (5,0) {$\delta \dashv \mathscr P$};
\end{tikzpicture}
\]

with the unit of the adjunction $\eta_X(x) = \delta_{\{x\}}$ and the counit  $\varepsilon_{X}\colon \px \to X$.
\end{theorem}
Thus, we can realize the Giry monad as the composition $T =  \mathscr P \,\circ\,  \delta$, and moreover, $\prob$ is equivalent to the smallest category through which $T$ factors---{\em i.e.}, it is equivalent to the Kleisli category $K(T)$ of the Giry monad. Hence every $\prob$ arrow $P\colon  X \to Y$ corresponds uniquely to a measurable arrow $X \to TY$.

 \section{Joint Distributions and Conditionals}
  Given a family of objects $\{X_i\}_{i \in I}$ we can form the cartesian product $\prod_{i \in I}X_i$ and endow this set with the product $\sigma$-algebra generated by all the projection maps $\prod_{i \in I}X_i \stackrel{\pi_j}{\longrightarrow} X_j$, one for each index $j \in I$.   It is easy to see that 
 \be
 	\left(\left(\prod_{i \in I}X_i,\bigotimes_{i\in I}\sa_{X_i}\right),\{\delta_{\pi_i}\}_{i \in I}\right)
\ee
does not give a categorical product. In fact, only weak products and equalizers exist in $\prob$, as the uniqueness condition fails for both constructions.  We use the terminology ``product space'' to denote the set product of any family $\{(X_i,\sa_{X_i})\}_{i \in I}$ of objects with the product $\sigma$-algebra  and not to imply that that object   
\[
	\left(\prod_{i \in I}X_i, \bigotimes_{i \in I} \sa_{X_i}\right)
\]
with projections satisfies any universality condition.   We will call a probability measure $J\colon  1 \to (\prod_{i \in I}X_i, \otimes_{i \in I} \sa_{X_i})$  on a product space a {\em joint distribution}. We do not mean to imply that these are distributions of a random variable, but rather indicate a measure on a product space which is not necessarily a product measure. These joint distributions are the main objects of study in Bayesian probability. 

Given any joint distribution $J\colon  1 \to \prod_{i \in I} X_{i}$, for each $j \in I$ we have the diagram
 \begin{equation}
 \label{productDiagram}
 \begin{tikzpicture}[baseline=(current bounding box.center)]
 	\node	(1)	at	(0,0)		         {$1$};
	\node	(X)	at	(-4,-3)	               {$X_j$};
	\node        (XY)  at      (0,-3)               {$\prod_{i \in I}X_i$};
	
	\draw[->,left,above] (1) to node  {$$} (X);
	\draw[->,right] (1) to node {$J$} (XY);
	\draw[->,left,auto] (XY) to node {$\delta_{\pi_{j}}$} (X);
	\draw (-2.9,-1.6) node {$\delta_{\pi_{j}}\circ J$};
 \end{tikzpicture}
 \end{equation}
where the composite $\delta_{\pi_{j}} \circ J$ is called the \emph{marginal} (distribution) \index{marginal} of $J$ on component $X_j$ and is given by $(\delta_{\pi_{j}} \circ J)(A_j) = J(\pi_{j}^{-1}A_j) $ by Lemma \ref{lem::Dirac_property}.

  Given only the probability measures on the components, $\{P_i\}_{i \in I}$,  there are many joint distributions on the product space whose marginals are the given family $\{P_i\}_{i \in I}$.    By using \emph{relationships} in the form of conditionals between the components, we bring into play additional knowledge that permits the determination of the appropriate joint distribution.  If the uncertainty of component $X_j$, as expressed by a  probability measure $P_j$ on component $X_j$, depends conditionally on a parameter which varies over component $X_i$ then we have the $\prob$ arrow $h\colon  X_i  \to X_j$.  These conditionals---which are the morphisms in $\prob$---are the key to determining a unique joint distribution.  The relationship between the components $X_i$ and $X_j$ is mediated by the conditional $h$ and expresses the relationship $P_j = h \circ P_i$.
  
 \subsection{Constructing a Joint Distribution Given Conditionals}
 \label{section::construct_joint}
 We now show how marginals and conditionals can be used to determine joint distributions in $\prob$. This development follows that of~\cite{Abramsky-Blute-Panangaden:nuclear_ideals} where we first learned of this approach (the category $\mathcal Stoch$ of stochastic kernels in that paper is nearly what we call $\prob$, although our restriction to countably generated spaces and perfect measures is replaced in the paper by restricting to Polish spaces).  Given a $\prob$-arrow $h\colon  X \to Y$ and a perfect probability measure  $P_{X}\colon  1 \to X$  on $X$, consider the diagram 
  \begin{equation}   \label{jointDefinition}
 \begin{tikzpicture}[baseline=(current bounding box.center)]
 	\node	(1)	at	(0,0)		         {$1$};
	\node	(X)	at	(-4,-3)	                  {$X$};
	\node	(Y)	at	(4,-3)               {$Y$};
	\node        (XY)  at      (0,-3)               {$X \times Y$};
	
	\draw[->, left, above] (1) to node  {$P_X$} (X);
	\draw[->,dashed,left,above] (XY) to node {$\delta_{\pi_{X}}$} (X);
	\draw[->,dashed,right,auto] (XY) to node {$\delta_{\pi_{Y}}$} (Y);
	\draw[->,dashed, below,auto] (1) to node {$J_{h}$} (XY);
	\draw[->]  (X) to [out=305,in=235,looseness=.5]  (Y) ;
	\draw (0,-3.9) node {$h$};
 \end{tikzpicture}
 \end{equation}
where $J_{h}$ is the uniquely determined joint distribution on the product space $X \times Y$ defined on the rectangles of the $\sigma$-algebra $\sa_{X} \otimes \sa_{Y}$ by
 \be \label{jointD}
 J_{h}(A \times B) = \int_{A} h_{B} \, dP_X.
 \ee
Then it follows from Theorem~\ref{theorem::perfect_measures} that $J_{h}$ is perfect.  The marginal of $J_{h}$ with respect to $Y$ then satisfies $\delta_{\pi_{Y}} \circ J_{h} = h \circ P_X$ and the marginal of $J_{h}$ with respect to $X$ is $P_X$, both of which are also perfect. By a symmetric argument, if we are given a probability measure $P_Y$  and conditional probability $k\colon  Y \to X$
 then we obtain a unique perfect joint distribution $J_{k}$ on the product space $X \times Y$ given on the rectangles by
 \be
 J_{k}(A \times B) = \int_{B} k_{A} \, dQ.
 \ee
However, if we are given $P_{X}, P_{Y}, h, k$ as indicated in the diagram
\be
 \begin{tikzpicture}[baseline=(current bounding box.center)]
 	\node	(1)	at	(0,0)		         {$1$};
	\node	(X)	at	(-3,-4)	                 {$X$};
	\node	(Y)	at	(3,-4)               {$Y$};
	\node        (XY)  at      (0,-2)               {$X \times Y$};
	
	\draw[->, above right] (1) to node  {$P_Y$} (Y);
	\draw[->, above left] (1) to node {$P_{X}$} (X);
	\draw[->,dashed, right] (XY) to node[xshift=8pt] {$\delta_{\pi_{X}}$} (X);
	\draw[->,dashed, left] (XY) to node {$\delta_{\pi_{Y}}$} (Y);
	\draw[->,dashed, right] ([xshift=2pt] 1.south) to node {$J_{k}$} ([xshift=2pt] XY.north);
	\draw[->,dashed, left] ([xshift=-2pt] 1.south) to node {$J_{h}$} ([xshift=-2pt] XY.north);
	\draw[->, above] ([yshift=2pt] X.east) to node {$h$} ([yshift=2pt] Y.west);
	\draw[->, below] ([yshift=-2pt] Y.west) to node {$k$} ([yshift=-2pt] X.east);

 \end{tikzpicture}
 \ee
then we have that $J_{h}=J_{k}$ if and only if the compatibility conditions \index{compatibility conditions}
 \be
 \begin{array}{lcl}
 P_X &=& k \circ P_Y \\
 P_Y &=& h \circ P_X
 \end{array}
 \ee
are satisfied.  Thus if the compatibility conditions are satisfied, then we can realize the product rule of probability in $\prob$ as
\be
\label{eqn::product_rule}
 \int_{A} h_{B} \, dP_X = J( A \times B ) = \int_{B} k_{A} \, dP_Y.
\ee
 In the extreme case, suppose we have a $\prob$-arrow $h\colon  X \to Y$ which factors through the terminal object $1$ as 
 \begin{equation}
 \begin{tikzpicture}[baseline=(current bounding box.center)]
 
 	\node	(X)	at	(0,0)		         {$X$};
	\node	(Y)	at	(3,0)	         {$Y$};
	\node	(1)	at	(1.5,-1)          {$1$};

	\draw[->, above] (X) to node  {$h$} (Y);
	\draw[->, below left] (X) to node {$!$} (1);
	\draw[->,below right]  (1)  to node {$Q$}  (Y);

 \end{tikzpicture}
 \end{equation}
 where $!$ represents the unique arrow from $X \to 1$. If we are also given a perfect probability measure $P\colon  1 \to X$, then we can calculate the joint distribution determined by $P$ and $h=Q \circ !$ as
\be
\begin{array}{lcl}
J(A \times B) &=& \int_{A} (Q \circ !)_B \, dP \\
&=& P(A) \cdot Q(B) 
\end{array}
\ee
so that  $J=P \otimes Q$.  This is precisely the situation where we say that the marginals $P$ and $Q$ are \emph{independent}.   Thus in $\prob$ independence corresponds to  a special instance of a conditional---one that factors through the terminal object.

\subsection{Constructing Regular Conditionals given a Joint Distribution} 
\label{section::constructing_regular_conditionals}
The following result is the basis from which the inference maps in Bayesian probability theory are subsequently constructed. The proof of the following theorem can be found in \cite{Faden:regular_conditional_probabilities}, where several equivalent conditions are identified. 

\begin{theorem}  \label{theorem::regular_conditionals} If $J: 1\to X \times Y$ is a $\prob$-morphism with marginals $P_{X}$ on $X$ and $P_{Y}$ on $Y$, then there exists a $\prob$-arrow $f$ that makes the diagram
  \begin{equation}
 \begin{tikzpicture}[baseline=(current bounding box.center)]
 	\node	(1)	at	(1.5,0)		         {$1$};
	\node	(Y)	at	(3,-2)	               {$Y$};
	\node        (X)  at      (0,-2)               {$X$};
	
	\draw[->,right] (1) to node  {$P_Y$} (Y);
	\draw[->,left] (1) to node {$P_{X}$} (X);
	\draw[->,dashed, below ] (Y) to node {$f$} (X); 
 \end{tikzpicture}
 \end{equation}
commute and satisfies
\be
	J(A \times B) = \int_{B} f_{A}\,dP_{Y}. 
\ee
Moreover, the morphism $f$ is the unique $\prob$-morphism with these properties,  up to a set of $P_Y$-measure zero.  
\end{theorem}

Interestingly, we can use Theorem~\ref{theorem::regular_conditionals} to obtain a seemingly stronger statement, {\em i.e.}, that the regular conditional probability factors through the product. Though this is not difficult to prove, we will prefer this stronger statement in the sequel.

\begin{theorem} \label{theorem::strong_regular_conditionals}
If $J: 1 \to X \times Y$ is a $\prob$-morphism with marginal distributions $P_X$ and $P_Y$ on $X$ and $Y$, then there exist  $\prob$ arrows $f$ and $g$
such that the diagram
  \begin{equation} \label{dgm::product_regular_conditionals}
 \begin{tikzpicture}[baseline=(current bounding box.center)]
 	\node	(1)	at	(0,0)		         {$1$};
	\node	(Y)	at	(3,-4)	               {$Y$};
	\node	(X)	at	(-3,-4)	               {$X$};
	\node        (XY)  at      (0,-2)               {$X \times Y$};
	
	\draw[->,right] (1) to node  {$P_Y$} (Y);
	\draw[->,left] (1) to node  {$P_X$} (X);
	\draw[->,left] (1) to node {$J$} (XY);
	\draw[->,above] (XY) to node {$\delta_{\pi_{Y}}$} (Y);
	\draw[->,above] (XY) to node {$\delta_{\pi_{X}}$} (X);
	\draw[->] (X) to [out=0,in=180,looseness=.5] (Y);
	\draw[->, below,out=210,in=-30,looseness=.5] (Y) to node {$\delta_{\pi_X} \circ f$} (X);
	\draw[->,dashed] (X) to [out=20,in=250,looseness=.5] (XY);
	\draw[->,dashed] (Y) to [out=160,in=290,looseness=.5] (XY);
         \draw (0.7, -3.2) node {$f$};
          \draw (-.7, -3.2) node {$g$};

         \draw (0, -3.7) node {$\delta_{\pi_Y} \circ g$};
 \end{tikzpicture}
 \end{equation}
commutes and 
\be
\int_A (\delta_{\pi_Y} \circ g)_B \, dP_X = J(A \times B) = \int_B (\delta_{\pi_X} \circ f)_A \,dP_Y.
\ee

\end{theorem}
\begin{proof}  We can apply Theorem~\ref{theorem::regular_conditionals} to see that there is a $\prob$-arrow $f \colon  Y \to X \times Y$ satisfying $J =f \circ P_Y$ such that
\be
	\int_{C}f_{A \times B} \,dP_{Y} = J\left(A \times (B \cap C)\right).
\ee
Then from Lemma~\ref{lem::Dirac_property}, we know that $(\delta_{\pi_X} \circ f)(y, A) = f(y, A \times Y)$ and so 
\begin{align}
\int_B (\delta_{\pi_X} \circ f)_A \,dP_Y &= \int_B   f_{A \times Y}  \,dP_Y \\
&= J\left(A \times (Y \cap B)\right) \\
&= J(A \times B)
\end{align}
Similarly we obtain a $\prob$-arrow $g\colon  X \to X \times Y$ satisfying $J =g\circ P_X$ and
\be
\int_A (\delta_{\pi_Y} \circ g)_B \,dP_X = J(A \times B). 
\ee
With these facts, it is a simple exercise to check that the diagram commutes.
\end{proof}

Note that if the joint distribution $J$ is obtained by a probability measure $P_X$ and a conditional $h\colon X \rightarrow Y$ using the method described by Diagram~\ref{jointDefinition}, then using the above result and notation it follows  $P_{X}$-a.s. that $h= \delta_{\pi_Y} \circ g$.

\begin{remark}  (Tonneli's Theorem) Given a $\prob$-morphism $J:1 \to X \times Y$, with marginals $P_X$ and $P_Y$ let $\gamma\colon  X \rightarrow X \times Y$ and $\varphi\colon  Y \rightarrow X \times Y$ be the $\prob$ arrows satisfying $\gamma \circ P_X = J$  and  $ \varphi \circ P_Y= J$ whose existence is guaranteed by Theorem~\ref{theorem::strong_regular_conditionals}.  Given any measurable function $F\colon  X \times Y \rightarrow [0,1]$ we have the  diagram
  \begin{equation}
 \begin{tikzpicture}[baseline=(current bounding box.center)]
 	\node	(1)	at	(0,-.8)		         {$1$};
	\node	(Y)	at	(3,-4)	               {$Y$};
	\node	(X)	at	(-3,-4)	               {$X$};
	\node        (XY)  at      (0,-4)               {$X \times Y$};
	\node        (2)    at      (0, -6)                {$2$};
	
	\draw[->,right] (1) to node  {$P_Y$} (Y);
	\draw[->,left] (1) to node  {$P_X$} (X);
	\draw[->,auto] (1) to node {$$} (XY);
	\draw[->,below] (XY) to node {$\delta_{\pi_{Y}}$} (Y);
	\draw[->,below] (XY) to node {$\delta_{\pi_{X}}$} (X);
	\draw[->] (X) to [out=30,in=150,looseness=.5] (XY);
	\draw[->] (Y) to [out=150,in=30,looseness=.5] (XY);
         \draw (1.3, -3.3) node {$\phi$};
          \draw (-1.3, -3.3) node {$\gamma$};
          \draw(.2,-2.75) node {$J$};
          \draw[->,auto] (XY) to node {$\overline{F}$} (2);
          \draw[->,dashed,left] (X) to node {$\overline{f} = \overline{F} \circ \gamma$} (2);
          \draw[->,dashed,right] (Y) to node[xshift=4pt] {$\overline{g} = \overline{F} \circ \varphi$} (2);

 \end{tikzpicture}
 \end{equation}
where the top two triangles commute. Thus we can define $\overline{f} = \overline{F} \circ \gamma$ and $\overline{g} = \overline{F} \circ \varphi$ so that the entire diagram commutes. From this, it follows that 
\be
\int_X f \, dP_X = \int_{X \times Y} F \, dJ = \int_Y g \, dP_Y
\ee
and we can realize Tonneli's Theorem as the special case with $J=P_X \otimes P_Y$.  
\end{remark}

This formulation provides the context for the following optimal transportation problem: given  marginals $P_X$ and $P_Y$ that model the supply and demand constraints, and a cost function $F$ (defined up to a scalar constant) representing the unit cost to transport a product from $x \in X$ to $y \in Y$, what joint distribution $J$ on $X \times Y$ with marginals $P_X$ and $P_Y$ minimizes the objective function $\int_{X \times Y}F \, dJ$?   The optimal assignment is then the conditional probability $X \rightarrow Y$ determined by the optimal joint distribution $J$. For example, this problem is investigated and a unique solution is given in certain cases in~\cite{Korman-McCann:optimal_transportation}.

\section{Bayesian probability in $\prob$}   

If we replace  $X$ and $Y$  in Diagram~\ref{dgm::product_regular_conditionals} by $D$(ata) and $H$(ypotheses), and the composites $\delta_{\pi_Y} \circ g$ by $\mathcal{S}$(ampling distribution) and  $\delta_{\pi_X} \circ f$ by $\mathcal{I}$(nference), then we can define $P_D = \mcS \circ P_H$ to obtain
 \begin{equation}
 \begin{tikzpicture}[baseline=(current bounding box.center)]
         \node         (1)    at      (0,2)         {$1$};
	\node	(H)	at	(-2,0)	      {$H$};
	\node	(D)	at	(2,0)               {$D$};	
	\draw[->,above left] (1) to node {$P_H$} (H);
	\draw[->,above right] (1) to node {$P_D$} (D);
	\draw[->, above] (H) to node {$\mathcal{S}$} (D);
         \draw[->, below,out=240, in = 300,looseness=.2] (D) to  node {$\mathcal I$} (H);
 \end{tikzpicture}.
 \end{equation}
In the context of Bayesian probability the probability measure $P_H$ is often called a \emph{prior probability}. \index{prior probability}

In this notation, the product rule in $\prob$ given in Equation~\ref{eqn::product_rule} becomes
 \be  \label{productRule}
 \int_{\mathcal{D}} \mathcal{I}_{\mathcal{H}} \, dP_D = J( \mathcal{H} \times \mathcal{D} ) = \int_{\mathcal{H}} \mathcal{S}_{\mathcal{D}} \, dP_H
 \ee
where $\mathcal{H} \in \sa_H$ and $\mathcal{D} \in \sa_D$.  We will spend the remainder of this sections showing how this interpretation of these spaces in $\prob$ provides a categorical foundation for Bayesian probability. First, we briefly review the fundamental concepts in Bayesian probability theory which can be found in \cite{Jaynes:probability} and then proceed to show how $\prob$ is the appropriate category for this theory. Generally, a {\em Bayesian model}  is comprised of a number of items including 
\begin{enumerate}
	\item[(i)]  two measurable spaces $H$ and $D$ representing hypotheses and data, respectively,
 	\item[(ii)] a probability measure $P_H$ on the $H$ space called the prior probability,
 	\item[(iii)] a $\prob$ arrow $\mathcal{S}\colon  H \rightarrow D$ called the sampling distribution,
 	\item[(iv)] a $\prob$ arrow $\mathcal{I}\colon  D \rightarrow H$ called the inference map,
\end{enumerate}
Note that the data space $D$ can also be thought of as the event space for some experiment, and the $\sigma$-algebra on $D$ is determined by distinguishable data. The prior probability measures are updated via the inference map as one takes measurements, which correspond to probability measures $\mu$ on $D$.  These updated probability measures are then called {\em posterior probabilities} \index{posterior probability} and are given by $\widehat{P}_H =\mcI \circ \mu$. The posterior $\widehat P_{H}$ then becomes the prior probability for the next step and the process continues as more measurements are taken. At each step in a Bayesian process, the posterior probability is a representation of knowledge about the hypotheses based on all of the data that has been accumulated up to that point.

Using the sampling distribution $\mathcal S$ and the prior probability $P_H$ on $H$, we can define a joint distribution $J$ on the product space $H \times D$ as in Section~\ref{section::construct_joint}
by defining it on the rectangles as  $J(A \times B) = \int_A \mcS_{B} \, dP_H$. The $D$-marginal (prior probability on data) is then  $P_D = \mcS \circ P_H =  \delta_{\pi_D} \circ J$. Using Theorem~\ref{theorem::strong_regular_conditionals}, we have the following theorem.
\begin{theorem}
\label{theorem::inference}
Given $\prob$ arrows representing a  prior probability $P_{H}\colon 1 \to H$ and a sampling distribution $\mathcal S\colon H \to D$, the inference map $\mathcal I\colon  D \to H$ is determined uniquely up to a set of $P_{D}$-measure zero.
\end{theorem}

\begin{proof}
Let $\mcI\colon  D \rightarrow H$ be the composition $\delta_{\pi_H} \circ f$, where $\delta_{\pi_{H}}$ is the projection $H \times D \to H$ and $f\colon  D \rightarrow H \times D$ is the $\prob$-arrow satisfying 
\be
J(U \times V) = \int_D f_{U \times V} \, d P_D
\ee
whose existence is given by Theorem~\ref{theorem::strong_regular_conditionals}. Thus
\be \label{uniqueI}
\int_A \mcS_B \, dP_H = J(A \times B) = \int_B \mcI_A \, dP_D
\ee
and this inference arrow $\mcI$ is unique in that if $\mcI'$ also satisfies equation~\ref{uniqueI} then the set
$\{ y \in Y \mid \mcI_y \ne \mcI'_y \}$ has $P_{D}$-measure zero.
\end{proof}

Thus the complete process works in the following way. A prior probability $P_{H}$ and sampling distribution $\mathcal S$ are specified, from which one determines the inference map $\mathcal I$. Once measurements $\mu\colon  1 \to D$ are taken, we then calculate the posterior probability by $\mathcal I \circ \mu$. This updating procedure can be characterized by the diagram
 \begin{equation} \label{BayesModel}
 \begin{tikzpicture}[baseline=(current bounding box.center)]
         \node         (1)    at       (0,2.5)              {$1$};
	\node	(H)	at	(-2.5,0)	      {$H$};
	\node	(D)	at	(2.5,0)               {$D$};
	\draw[->,above left] (1) to node {$P_H$} (H);
	\draw[->,dashed,above right] (1) to node {$\mu$} (D);
	\draw[->, above] (H) to node {$\mathcal S$} (D);
         \draw[->, densely dotted, below,out=240, in = 300,looseness=.2 ] (D) to node {$\mathcal I$}  (H);
	\draw[->,dashed,right, out=255,in=30] (1) to node[xshift=5pt] {$\mathcal I \circ \mu$} (H);
 \end{tikzpicture}
 \end{equation}
 where the solid lines indicate arrows given {\em a priori}, the dotted line indicates the arrow determined using Theorem~\ref{theorem::strong_regular_conditionals}, and the dashed lines indicate the updating after a measurement. Note that if there is no uncertainty in the measurement, then $\mu=\delta_{\{x\}}$ for some $x \in D$, but in practice there is usually some uncertainty in the measurements themselves.
 
 Following the calculation of the posterior probability, the sampling distribution is then updated, if required. The process can then repeat:  using the posterior probability and the updated sampling distribution the updated joint probability distribution on the product space is determined and the corresponding (updated) inference map determined.  We can then continue to iterate as long as new measurements are received.  For some problems (such as with the standard urn problem with replacement of balls) the sampling distribution does not change from iterate to iterate, but the inference map is updated since the posterior probability on the hypothesis space changes with each measurement. The model selection problem (either once at the beginning of this process, or iteratively throughout) can also be modeled as a meta-Bayesian process, where the hypothesis space is the space of potential models and the data constitutes some information that would inform on the suitability of a given model. 
\begin{remark}
We know from Theorem~\ref{theorem::inference} that the inference map $\mathcal I$ is uniquely determined by $P_{H}$ and $\mathcal S$ up to a set of $P_{D}$-measure zero. However, there is no reason {\em a priori} that a measurement $\mu\colon  1 \to D$ is required to be absolutely continuous with respect to $P_{D}$. In $\mu$ is not absolutely continuous with respect to $P_{D}$, then a different choice of inference map $\mathcal I'$ could yield a different posterior probability---{\em i.e.}, we could have $\mathcal I \circ \mu \neq \mathcal I' \circ \mu$. Thus we make the assumption that measurement probabilities on $D$ are absolutely continuous with respect to the prior probability $P_{D}$ on $D$. This is a reasonable assumption, however, since if a data event is impossible (has $P_{D}$-measure zero) under a certain model, then the model should not be expected to make an meaningful inference when presented with that data. On the other hand, it is easy to see that if a measurement $\mu \ll P_{D}$, then $\mathcal I \circ \mu \ll P_{H}$, as expected.
\end{remark}  

  We emphasize that this procedure can be employed for any perfect prior probability and any regular conditional probability.  For example, given a perfect prior $P\colon  1 \to \px$ and the conditional $\varepsilon_{X}\colon  \px \to X$ there corresponds a unique inference map $\mcI\colon  X \to \px$ satisfying, for all $A \in \sa_X$ and for all $\B \in \sa_{\px}$,
\be
\int_{\px}{\varepsilon_{X}}_A \, dP = \int_{X} \mcI_{\B} \, d(\varepsilon_{X} \circ P).
\ee
  
In the case where $X = 2 = \{\top, \bot\}$ (the two element set), these ``higher order distributions''  $P\colon  1 \to \mathscr P2$
can be used to explicate the concept of $A_p$ distributions as characterized by Jaynes~\cite[Chapter 18]{Jaynes:probability}. 
Using our notation, a proposition $A$, which is a morphism $A\colon  1 \to 2$ in the category $\mathcal Set$ of sets,  has an associated probability of truth, say $Pr(A)=p$. Hence $A$ determines a $\prob$-morphism $\overline{A}\colon  1 \to 2$, with $\overline{A}(\{\top\})=p$.
The information supplied by the arrow $\overline{A}$ consists only of the single value $p$ and fails to indicate how sensitive this proposition is to additional data. 
The confidence that one has in the value $p$ can be supplied by the higher order distributions which are probability measures on the space $\mathscr P2$ of probability measures on $2$.  Since $\mathscr P2$ consists precisely of the Bernoulli distributions $B_{\theta} = \theta \delta_{\top} + (1- \theta) \delta_{\bot}$, where $\theta \in [0,1]$, it follows that  $\mathscr P2 \simeq [0,1]$. Consequently, any distribution
\be
1 \xrightarrow{A_p} \mathscr P2
\ee
has an expected value which can be calculated using the composition
 \be
1 \xrightarrow{A_p} \mathscr P2 \xrightarrow{\varepsilon_{2}} 2.
\ee
Thus $E(A_{p}) = (\varepsilon_{2} \circ A_p)( \{\top\}) = p$ and any such distribution provides a more informative measure.  For example, the two distributions on $\mathscr P2 \simeq [0,1]$ specified by  $p = \delta_{\frac{1}{2}}$ and $p'$ the uniform (Lebesque) measure
both have expected value $\frac{1}{2}$. Yet clearly, the first is deterministic, expressing a (complete) confidence in the statement that the expected value of the proposition 
\begin{equation}
1 \xrightarrow{~~\overline{A}~~} 2
 \end{equation}
where $\overline{A}(\top) = \frac{1}{2}$ is $\frac{1}{2}$. On the other hand, the distribution $p'$ also determines $\overline{A}$, but instead expresses a maximal ignorance modeled by the uniform distribution.

\section{The Category of Decision Rules}

 Recall that the Giry monad $T$ factors through $\prob$ via the adjunction of Theorem~\ref{adjunction}. At the other end of the spectrum of categories through which the Giry monad factors is the Eilenberg--Moore category $\measT$, consisting of the Eilenberg--Moore algebras of the Giry monad. From the theory of monads (see \cite{Barr-Wells:TTT}, for example), we know that $\prob$ then embeds into $\measT$, which has additional structure that is useful for dealing with other aspects of probability theory and decision making which $\prob$ is not equipped for.  Let us briefly recall the definition of a $T$-algebra. 

If  $(T,\eta,\mu)$ is a monad in a category $\mathcal{C}$, a {\em $T$-algebra} $(X,\alpha)$ is a pair consisting of an object $X$ in $\mathcal{C}$ and a $\mathcal{C}$ arrow $\alpha\colon  TX \to X$ such that the diagrams
\begin{equation}   \label{associativity}
 \begin{tikzpicture}[baseline=(current bounding box.center)]
 
        \node      (TTX)   at    (-7,0)      {$T^2X$};
        \node       (T1X)  at    (-7,-2)     {$TX$};
        \node       (T2X)  at    (-5, 0)     {$TX$};
        \node       (X0)    at    (-5,-2)     {$X$};
        \draw[->,left] (TTX) to node {$T\alpha$} (T1X);
        \draw[->,above] (TTX) to node {$\mu_X$} (T2X);
        \draw[->,right](T2X) to node {$\alpha$} (X0);
        \draw[->,below](T1X) to node {$\alpha$} (X0);
        
         \node         (X1)  at      (-2,0)      {$X$};
 	\node	(TX)	at	(0,0)		         {$TX$};
	\node	(X)	at	(0,-2)	         {$X$};	
	\draw[->,right] (TX) to node {$\alpha$} (X);
	\draw[->,above] (X1) to node {$\eta_X$} (TX);
	\draw[->,below left ] (X1) to node {$Id$} (X);		
 \end{tikzpicture}
 \end{equation}
commute; the first diagram is called the associative law and the second diagram the unit law.  A morphism of $T$-algebras $(X,\alpha) \stackrel{f}{\longrightarrow} (Y,\beta)$ is an arrow $f\colon X \to Y$ of $\mathcal{C}$ such that the diagram
 \begin{equation}   \label{morphism}
 \begin{tikzpicture}[baseline=(current bounding box.center)]
 	\node	(X1)	at	(0,0)		         {$TX$};
	\node	(Y1)	at	(3,0)	                  {$TY$};
	\node	(X)	at	(0,-2.5)	         {$X$};
	\node        (Y)  at      (3,-2.5)                     {$Y$};
	
	\draw[->,above] (X1) to node  {$Tf$} (Y1);
	\draw[->,right] (Y1) to node  {$\beta$} (Y);
	\draw[->,left] (X1) to node {$\alpha$} (X);
	\draw[->,below] (X) to node {$f$} (Y);
		
 \end{tikzpicture}
 \end{equation}
commutes.

When $T$ is the Giry monad, an algebra $TX \stackrel{\alpha}{\longrightarrow} X$ consists of a measurable space $X$ with a countably generated $\sigma$-algebra, the space $TX = \px$ of probability measures on $X$, and a measurable map $\alpha$ satisfying the two defining properties of a $T$-algebra.  The measurable map $Tf$ in the definition of a morphism of $T$-algebras is the pushforward map:  given $P \in TX$ the pushforward by $f$ is $Tf(P) = f_{\ast}P  \in \mathscr PY$.
 
 The $T$-algebras $(X, \alpha)$ are often called decision rules since the measurable map $\alpha$ assigns (decides) a value in $X$ to each probability measure $P$ on $X$.  Alternatively, we can think of a decision rule as collapsing a probability distribution to a definite value, or {\em derandomizing} a probability distribution as in \cite{Doberkat:tracing}.  For this reason, we often use the descriptive characterization of \v{C}encov~\cite{Cencov:statistical_decision_rules}  and call the category $\measT$ the {\em category of decision rules}.\footnote{\v{C}encov did not work in $\measT$, but rather in $\prob$, restricting to measurable spaces $X$ such that $\px$ has a $\sigma$-algebra generated by finitely many atoms. The primary difference in the current approach and that of \v{C}encov is that we take a Bayesian viewpoint, while he is attempting to describe the standard statistical inference perspective.}

Embedding the main consequence of the existence of regular conditional probabilities for Bayesian probability into $\measT$, we have the following.
\begin{theorem}
Given a measurable function $\mcS\colon  H \rightarrow TD$, there exists a measurable function $\mcI\colon  D \to TH$ such that  $\hat{\mcI}=\mu_H \circ \mcI$ is a retraction of $\hat{\mcS} =\mu_{D} \circ T\mcS $ in $\M$. 
\end{theorem}

\begin{proof}
This follows immediately from Theorem~\ref{theorem::inference} and the embedding of $\prob$ into $\measT$ and can be summarized in the diagram
  \begin{equation}   \label{associativity}
 \begin{tikzpicture}[baseline=(current bounding box.center)]
 	\node	(H)	at	(0,0)		         {$TH$};
	\node	(D2)	at	(3,1.5)	                  {$T^2D$};
	\node         (D)   at      (6,0)                   {$TD$};
	\node         (H2)  at     (3,-1.5)                  {$T^2H$};
	
	\draw[->,above left] (H) to node  {$T\mcS$} (D2);
	\draw[->,above right] (D2) to node  {$\mu_D$} (D);
	\draw[->,below right] (D) to node {$T\mcI$} (H2);
	\draw[->,below left] (H2) to node {$\mu_H$} (H);
	
	\draw[->, above] ([yshift=3pt] H.east) to node {$\hat{\mcS}$} ([yshift=3pt] D.west);
	\draw[->, below] ([yshift=-3pt] D.west) to node {$\hat{\mcI}$} ([yshift=-3pt] H.east);
 \end{tikzpicture}.
\end{equation}
\end{proof}

One of the primary advantages to using the existence of regular conditional probabilities guaranteed by Theorem~\ref{theorem::strong_regular_conditionals} is that we only require that measurable spaces have countably generated $\sigma$-algebras and that the measures are perfect. In contrast, much of the previous work involving the Giry monad and regular conditional probabilities requires resorting to topological arguments and restricting to Polish spaces. For example, in \cite{Doberkat:stochastic_relations}, Doberkat characterizes the $T$-algebras for the Giry monad under the Polish space assumption, and proves the counter-intuitive result that there are no non-trivial decision rules for finite spaces in $\measT$. In contrast, we exhibit a finite space having a $T$-algebra when one does not require topological restrictions.  

\begin{example}
An important such case is the decision rule  $d\colon T2 \rightarrow 2$ given by 
\be \label{subobject}
d(P)= 
\begin{cases}
\top & \text{if } P(\{\top\})=1 \\
\bot & \text{if } P(\{\top\}) < 1.
\end{cases}
\ee
The function $d$ is measurable since $d^{-1}(\{\top\}) = \{ \delta_{\top}\} \in \sa_{T2}$.   The associativity identity 

  \begin{equation}
 \begin{tikzpicture}[baseline=(current bounding box.center)]
 	\node	(22)	at	(0,0)		         {$T^22$};
	\node	(21)	at	(3,0)	                  {$T2$};
	\node	(12)	at	(0,-3)	         {$T2$};
	\node        (2)  at      (3,-3)                     {$2$};
	
	\draw[->,above] (22) to node  {$Td$} (21);
	\draw[->,right] (21) to node  {$d$} (2);
	\draw[->,left] (22) to node {$\mu_2$} (12);
	\draw[->,below] (12) to node {$d$} (2);
	
	\node	(Q)	at	(5,0)		         {$Q$};
	\node	(d1Q)	at	(9,0)	                  {$Qd^{-1}$};
	\node	(mu)	at	(5,-3)	         {$\mu_2(Q)$};
	\node        (tdd)  at      (9,-3)                     {$d( \mu_2(Q)) = d(Qd^{-1})$};
	
	\draw[->,above] (Q) to node  {$Td$} (d1Q);
	\draw[->,right] (d1Q) to node  {$d$} (tdd);
	\draw[->,left] (Q) to node {$\mu_2$} (mu);
	\draw[->,below] (mu) to node {$d$} (tdd);

 \end{tikzpicture}
 \end{equation}
where $\mu$ is the monad multiplication defined by 
\begin{equation}
\mu_2(Q)(A) = \int_{q \in T(2)} ev_A(q) \, dQ,
\end{equation}
is satisfied since both routes map the element $Q \ne \delta_{\delta_{\top}} \in T^2(2)  \mapsto \bot$ while $\delta_{\delta_{\top}} \mapsto \top$.  The unit law $Id_2 = d \circ \eta_2$ is trivial to verify.

The decision rule $d\colon  T2 \rightarrow 2$ partitions the space $T2$ into 
$\delta_{ \top}$ and all measures on  $2$ whose value on $\{ \top \}$ is of measure less than one.  There are many other decision rules for $2$, and any other finite or nonfinite space. Characterizing decision rules without the requirement for continuity is an open problem.  
\end{example}

\section{Acknowledgements}
The authors would like to express gratitude to F.W. Lawvere, P.F. Stiller and T. Nguyen for many fruitful conversations regarding these ideas. We also thank the reviewer for the helpful comments provided that have led to a much clearer exposition. This work was partially supported by the Air Force Office of Scientific Research, for which the authors are extremely grateful.


\begin{thebibliography}{10}
\providecommand{\url}[1]{{#1}}
\providecommand{\urlprefix}{URL }
\expandafter\ifx\csname urlstyle\endcsname\relax
  \providecommand{\doi}[1]{DOI~\discretionary{}{}{}#1}\else
  \providecommand{\doi}{DOI~\discretionary{}{}{}\begingroup
  \urlstyle{rm}\Url}\fi

\bibitem{Abramsky-Blute-Panangaden:nuclear_ideals}
Abramsky, S., Blute, R., Panangaden, P.: Nuclear and trace ideals in tensored
  $\ast$-categories.
\newblock J. Pure Appl. Algebra \textbf{143}, 3--47 (1999)

\bibitem{Barr-Wells:TTT}
Barr, M., Wells, C.: Toposes, triples and theories.
\newblock Repr. Theory Appl. Categ. (12) (2005).
\newblock Corrected reprint of the 1985 original (Springer-Verlag)

\bibitem{Berger:Statistical_decision_theory}
Berger, J.: Statistical Decision Theory and Bayesian Analysis, 2nd edn.
\newblock Springer Series in Statistics. Springer-Verlag, New York (1985)

\bibitem{vanBreugel:metric_monad}
van Breugel, F.: The metric monad for probabilistic nondeterminism (2005).
\newblock Unpublished

\bibitem{Doberkat:stochastic_relations}
Doberkat, E.E.: Characterizing the {E}ilenberg--{M}oore algebras for a monad of
  stochastic relations (2004).
\newblock Internal Memorandum No. 147

\bibitem{Doberkat:tracing}
Doberkat, E.E.: Derandomizing probabilistic semantics through
  {E}ilenberg--moore algebras for the {G}iry monad (2004).
\newblock Internal Memorandum No. 149

\bibitem{Doberkat:Kernel}
Doberkat, E.E.: Kleisli morphisms and randomized congruences for the giry
  monad.
\newblock J. Pure Appl. Algebra \textbf{211}(3), 638--664 (2007)

\bibitem{Dudley:probability}
Dudley, R.: Real Analysis and Probability.
\newblock No.~74 in Cambridge Studies in Advanced Mathematics. Cambridge
  University Press, Cambridge (2002)

\bibitem{Faden:regular_conditional_probabilities}
Faden, A.M.: The existence of regular conditional probabilities: necessary and
  sufficient conditions.
\newblock Ann. Probab. \textbf{13}(1), 288 -- 298 (1985)

\bibitem{Giry:categorical_probability}
Giry, M.: A categorical approach to probability theory.
\newblock Categorical Aspects of Topology and Analysis \textbf{915}, 68--85
  (1981)

\bibitem{Howson-Urbach:scientific_reasoning}
Howson, C., Urbach, P.: Scientific Reasoning: The Bayesian Approach, 2nd edn.
\newblock Open Court Publishing, Chicago (1993)

\bibitem{Jaynes:probability}
Jaynes, E.T.: Probability Theory: The Logic of Science.
\newblock Cambridge University Press, Cambridge (2003)

\bibitem{Kock:commutative_monads}
Kock, A.: Commutative monads as a theory of distributions.
\newblock Theory Appl. Categ. \textbf{26}, 97--131 (2012)

\bibitem{Korman-McCann:optimal_transportation}
Korman, J., Mc{C}ann, R.: Optimal transportation with capacity constraints
  (2012).
\newblock Ar{X}iv:1201.6404v2

\bibitem{Lawvere:prob_mappings}
Lawvere, W.F.: The category of probabilistic mappings (1962).
\newblock Unpublished

\bibitem{Lawvere:personal_2011}
Lawvere, W.F.: Bayesian sections (2011).
\newblock Personal communication

\bibitem{Meng:metric_spaces}
Meng, X.: Categories of convex sets and of metric spaces, with applications to
  stochastic programming and related areas.
\newblock Ph.D. thesis, State University of New York at Buffalo (1988)

\bibitem{Pachl:disintegrations}
Pachl, J.K.: Disintegration and compact measures.
\newblock Math. Scand. \textbf{43}, 157 --168 (1978)

\bibitem{Rodine:perfect_probabilities}
Rodine, R.H.: Perfect probability measures and regular conditional probabilities.
\newblock Ann. Math. Statist. \textbf{37}, 1273 -- 1278 (1966)

\bibitem{Ryll-Nardzewski:quasi-compact_measures}
Ryll-{N}ardzewski, C.: On quasi-compact measures.
\newblock Fund. Math. \textbf{40}, 125--130 (1953)

\bibitem{Cencov:statistical_decision_rules}
\v{C}encov, N.: Statistical Decision Rules and Optimal Inference,
  \emph{Translations of Mathematical Monographs}, vol.~53.
\newblock American Mathematical Society (1982)

\bibitem{Wendt:disintegrations}
Wendt, M.: The category of disintegrations.
\newblock Cahiers Topologie G\'{e}om. Diff\'{e}rentielle Cat\'{e}g.
  \textbf{35}(4), 291--308 (1994)

\end{thebibliography}

\end{document}